\documentclass[11pt]{amsart}

\usepackage{amsfonts,amsmath,amssymb,color,amscd,amsthm}
\usepackage[alphabetic]{amsrefs}
\usepackage[T1]{fontenc}
\usepackage{typearea}

\newtheorem{theorem-intro}[equation]{Theorem}
\newtheorem{proposition-intro}[equation]{Proposition}
\newtheorem{corollary-intro}[equation]{Corollary}
\newtheorem{question-intro}[equation]{Question}
\newtheorem{problem-intro}[equation]{Problem}

\newtheorem{lemma}{Lemma}[section]
\newtheorem{corollary}[lemma]{Corollary}
\newtheorem{theorem}[lemma]{Theorem}
\newtheorem{proposition}[lemma]{Proposition}

\newtheorem{question}[lemma]{Question}

\theoremstyle{definition}
\newtheorem{definition}[lemma]{Definition}
\newtheorem{notation}[lemma]{Notation}

\theoremstyle{remark}
\newtheorem{remark}[lemma]{Remark}

\newcommand\kk{{\mathbf k}}

\newcommand\N{{\mathbb N}}

\newcommand\id{{\mathrm{id}}}

\newcommand\A{{\mathbb A}}
\newcommand\Z{{\mathbb Z}}

\newcommand\C{{\mathbb C}}

\newcommand\Acal{{\mathcal A}}
\newcommand\Bcal{{\mathcal B}}
\newcommand\Fcal{{\mathcal F}}
\newcommand\Gcal{{\mathcal G}}
\newcommand\Ocal{{\mathcal O}}

\newcommand\Spec{{\mathrm{Spec}}}

\newcommand\SL{{\mathrm{SL}}}
\newcommand\PSL{{\mathrm{PSL}}}
\newcommand\End{{\mathrm{End}}}
\newcommand\GL{{\mathrm{GL}}}

\newcommand\Aut{{\mathrm{Aut}}}
\newcommand\SAut{{\mathrm{SAut}}}

\newcommand\tr{\hbox to 1mm  {${}^t \!  $} }

\DeclareMathOperator{\Stab}{Stab}
\DeclareMathOperator{\Jac}{Jac}

\DeclareMathOperator{\B}{B}
\DeclareMathOperator{\V}{V}
\DeclareMathOperator{\ev}{ev}
\DeclareMathOperator{\mdeg}{mdeg}

\title{On the maximality of the triangular subgroup}
\thanks{The authors gratefully acknowledge support by the Swiss National Science Foundation Grant  "Birational Geometry" PP00P2\_128422 /1 and by the French National Research Agency Grant "BirPol", ANR-11-JS01-004-01.}

\author[J.-P.~Furter]{Jean-Philippe Furter}
\address{Jean-Philippe Furter, Universit\"{a}t Basel, Spiegelgasse $1$, CH-$4051$ Basel, Switzerland\\ On leave from Dpt. of Math., Univ. of La Rochelle, av. Cr\'epeau, 17000 La Rochelle, France}
\email{jpfurter@univ-lr.fr}

\author[P.-M.~Poloni]{Pierre-Marie~Poloni}
\address{Pierre-Marie Poloni \\ Universit\"{a}t Bern \\ Mathematisches Institut \\Sidlerstrasse 5 \\ CH-3012 Bern \\ Switzerland}
\email{pierre.poloni@math.unibe.ch}

\begin{document}

\begin{abstract}
We prove that the subgroup of triangular automorphisms of the complex affine $n$-space is maximal among all solvable subgroups of $\Aut(\A_{\C}^n)$ for every $n$. In particular, it is a Borel subgroup of $\Aut(\A_{\C}^n)$, when the latter is viewed as an ind-group. In dimension two, we prove that the triangular subgroup is a  maximal closed subgroup and that nevertheless, it is not  maximal among all subgroups of $\Aut(\A_{\C}^2)$. Given an automorphism $f$ of $\A_{\C}^2$, we study the question whether the group generated by $f$ and the triangular subgroup is equal to the whole group $\Aut(\A_{\C}^2)$.
\end{abstract}

\maketitle

\section{Introduction}\label{section:intro}

The main purpose of this paper is to study the Jonqui\`eres subgroup $\Bcal_n$ of the group $\Aut(\A_{\C}^n)$ of polynomial automorphisms of the complex affine $n$-space, i.e.\ its subgroup of triangular automorphisms. We will  settle the  titular question by providing three different answers, depending on to which properties the maximality condition is referring to. 

\begin{theorem-intro}  \label{main-theorem}\begin{enumerate}
\item For every $n\geq2$, the subgroup $\Bcal_n$   is maximal among all solvable subgroups of $\Aut(\A_{\C}^n)$.
\item The subgroup  $\Bcal_2$ is  maximal among the closed subgroups of $\Aut(\A_{\C}^2)$. 
\item The subgroup  $\Bcal_2$ is not  maximal among all subgroups of $\Aut(\A_{\C}^2)$.
\end{enumerate}
\end{theorem-intro}

Recall that  $\Aut(\A_{\C}^n)$ is naturally an ind-group, i.e.\ an infinite dimensional algebraic group. It is thus  equipped with the usual ind-topology (see Section~\ref{section:ind-groups} for the definitions). In particular, since $\Bcal_n$ is a closed connected solvable subgroup of  $\Aut(\A_{\C}^n)$,  the first statement of Theorem \ref{main-theorem} can be  interpreted as follows:

\begin{corollary-intro}\label{cor:Borel}The group $\Bcal_n$ is a Borel subgroup of $\Aut(\A_{\C}^n)$.
\end{corollary-intro}

This generalizes a remark of  Berest, Eshmatov and Eshmatov  \cite{BEE}  stating  that triangular automorphisms of $\A^2_{\C}$, of Jacobian determinant  $1$, form a Borel subgroup (i.e.\ a maximal connected solvable subgroup) of the group $\SAut(\A_{\C}^2)$ of polynomial automorphisms of $\A_{\C}^2$ of Jacobian determinant $1$. Actually, the proofs in \cite{BEE} also imply Corollary \ref{cor:Borel} in the case $n=2$. Nevertheless, since they are based on results of Lamy \cite{Lamy}, which use the Jung-van der Kulk-Nagata structure theorem for $\Aut(\A_{\C}^2)$, these arguments are specific  to the dimension $2$ and cannot be generalized to higher dimensions. 

The Jonqui\`eres subgroup of $\Aut(\A_{\C}^n)$ is thus a good analogue of the subgroup of invertible upper triangular matrices, which is a Borel subgroup of the classical linear algebraic group $\GL_n(\C)$. Moreover, Berest, Eshmatov and Eshmatov strengthen this analogy when $n=2$   by proving that $\Bcal_2$ is, up to conjugacy, the only Borel subgroup of $\Aut(\A_{\C}^2)$.  On the other hand, it is well known that there exist, if $n\geq3$, algebraic additive group actions on $\A_{\C}^n$ that cannot be triangularized \cites{Bass,Popov87}. Therefore, we ask the following problem.

\begin{problem-intro}
Show that Borel subgroups of $\Aut(\A_{\C}^n)$ are not all conjugate ($n\geq3$).
\end{problem-intro}  

This problem turns out to be closely related to the question of the  boundedness of the derived length of solvable subgroups of $\Aut(\A_{\C}^n)$. We  give such a bound when $n=2$. More precisely, the maximal derived length of a solvable   subgroup of $\Aut(\A_{\C}^2)$ is equal to $5$ (see Proposition~\ref{prop: bound for the derived length of solvable subgroups of Aut(A2)}). As a  consequence, we prove that the group $\Aut_z(\A^3_{\C})$ of automorphisms of $\A^3$ fixing the last coordinate admits non-conjugate Borel subgroups (see Corollary~\ref{corollary: Aut_z(A^3) admits non-conjugate Borel subgroups}). Note that such a phenomenon has already been pointed out in \cite{BEE}.

\medskip

The paper is organized as follows. Section \ref{section:intro} is the present introduction. In Section \ref{section:ind-groups}, we recall the definitions of ind-varieties and ind-groups given by Shafarevich and explain how the automorphism group of the affine $n$-space may be endowed with the structure of an ind-group.

In Section \ref{section:Borel-subgps},  we prove the first two statements of Theorem \ref{main-theorem} and discuss the question, whether the ind-group $\Aut(\A_{\C}^n)$ does admit non-conjugate Borel subgroups. We then study the group of all automorphisms of $\A_{\C}^3$ fixing the last variable, proving that it admits non-conjugate Borel subgroups. In the last part of Section \ref{section:Borel-subgps}, we give examples of maximal closed subgroups of $\Aut(\A_{\C}^n)$.

 Finally, we consider $\Aut(\A_{\C}^2)$ as an ``abstract'' group in Section \ref{section:dim2}. We show that triangular automorphisms do not form a maximal subgroup of $\Aut(\A_{\C}^2)$.
More precisely, after defining the affine length of an automorphism in Definition~\ref{definition: affine length}, we prove the following statement:

\begin{theorem-intro}
For any field $\kk$, the two following assertions hold.
\begin{enumerate}
\item If the affine length of an automorphism $f \in \Aut (\A^2_{\kk})$ is at least $1$ (i.e. $f$ is not triangular) and at most $4$, then the group generated by $\Bcal_2$ and $f$ satisfies
\[ \langle \Bcal_2, f \rangle =  \Aut (\A^2_{\kk}).\]
\item There exists an automorphism $f \in \Aut (\A^2_{\kk})$ of affine length $5$ such that the group $\langle \Bcal_2, f \rangle$ is strictly included into $\Aut (\A^2_{\kk})$.
\end{enumerate} 
\end{theorem-intro}

\bigskip

{\bf Acknowledgement.} The authors thank the referee for helpful comments and suggestions which helped to improve the text.


\section{Preliminaries: the ind-group of polynomial automorphisms}  \label{section:ind-groups}


In \cites{Shafarevich1, Shafarevich2}, Shafarevich introduced the notions of ind-varieties and ind-groups, and explained how to endow the group  of polynomial automorphisms of the affine $n$-space with the structure of an ind-group. Since these two papers are well-known to contain several inaccuracies, we now recall the definitions from Shafarevich and describe the ind-group structure   of the automorphism group of the affine $n$-space.

For simplicity, we  assume in this section that $\kk$ is an algebraically closed field.

\subsection{Ind-varieties and ind-groups}

We first define the category of infinite dimensional algebraic varieties (\emph{ind-varieties} for short).

\begin{definition}[Shafarevich, 1966]\hfill
\begin{enumerate}
\item
An ind-variety $V$ (over $\kk$) is a set together with an ascending filtration $V_{\leq \, 0} \subseteq V_{\leq \, 1} \subseteq V_{\leq \, 2 }  \subseteq \cdots  \subseteq V$ such that the following holds:
\begin{enumerate}
\item $V= \bigcup_d V_{\leq \, d}$.
\item Each $V_{\leq \, d}$ has the structure of an algebraic variety (over $\kk$).
\item Each $V_{\leq \, d}$ is Zariski closed in $V_{\leq \, d+1}$.
\end{enumerate}

\item A morphism of ind-varieties  (or ind-morphism)  is a map $\varphi \colon V \to W$ between two ind-varieties $V= \bigcup_d V_{\leq \, d}$ and $W= \bigcup_d W_{\leq \, d}$ such that  there exists, for every $d$, an $e$ for which $\varphi (V_{\leq \, d} ) \subseteq W_{\leq \, e}$ and such that the induced map $V_{\leq \, d} \to W_{\leq \, e}$ is a morphism of varieties (over $\kk$).
\end{enumerate}
\end{definition}

In particular, every ind-variety $V$ is naturally equipped with the so-called ind-topology in which a subset $S \subseteq V$ is closed if and only if every subset $S_{\leq \, d} := S \cap V_{\leq \, d}$ is Zariski-closed in $V_{\leq \, d}$.

We remark that the product $V\times W$ of two ind-varieties  $V= \bigcup_d V_{\leq \, d}$ and $W= \bigcup_d W_{\leq \, d}$ has the structure of an ind-variety for the filtration $V\times W= \bigcup_d V_{\leq \, d}\times W_{\leq \, d}$. 

\begin{definition}An \emph{ind-group} is a group $G$ which is an ind-variety such that the multiplication $G\times G\to G$ and inversion $G\to G$ maps are morphisms of ind-varieties. 
\end{definition}

If $G$ is an abstract group, we denote by $D(G) = D^1(G)$  its (first) derived subgroup. It is the subgroup generated by all commutators $[g,h] := ghg^{-1}h^{-1}$, $g,h \in G$. The $n$-th derived subgroup of $G$ is then defined inductively by $D^n(G) = D^1 ( D^{n-1}(G))$ for $n \geq 1$, where by definition $D^0 (G) = G$. A group $G$ is called \emph{solvable} if $D^n(G) = \{ 1 \}$ for some integer $n \geq 0$. Furthermore, the smallest such integer $n$ is called the \emph{derived length} of $G$.

For later use, we  state (and prove) the following results which are well-known for algebraic groups and which extend straightforwardly to ind-groups.

\begin{lemma} \label{lemma: the closure of a (solvable) subgroup is a (solvable) subgroup}
Let $H$ be a subgroup of an ind-group $G$. Then, the following assertions hold.
\begin{enumerate}
\item \label{the closure of a  subgroup is a subgroup} The closure $\overline{H}$ of $H$ is again a subgroup of $G$.
\item \label{the derivation of the closure is included into the closure of the derivation} We have $D (\overline{H} ) \subseteq \overline{ D(H) }$.
\item \label{the closure of a solvable subgroup is a solvable subgroup} If $H$ is solvable, then $\overline{H}$ is solvable too.
\end{enumerate}
\end{lemma}

\begin{proof}

(\ref{the closure of a  subgroup is a subgroup}). The proof for algebraic groups given in \cite[Proposition 7.4A, page 54]{Hum81} directly applies to ind-groups. This proof being very short, we  give it here. Inversion being a homeomorphism, we get $(\overline{H})^{-1} = \overline{ H^{-1} } = \overline{H}$. Similarly, left translation by an element $x$ of $H$ being a homeomorphism, we get $x \overline{H} = \overline{xH} = \overline{H}$, i.e. $H \overline{H} \subseteq \overline{H}$. In turn, right translation by an element $x$ of $\overline{H}$ being a homeomorphism, we get $\overline{H} x= \overline{Hx} \subseteq \overline{H \overline{H}} \subseteq \overline{ \overline{H}} = \overline{H}$. This says that $\overline{H}$ is a subgroup.

(\ref{the derivation of the closure is included into the closure of the derivation}). Fix an element $y$ of $H$. The map $\varphi \colon G \to G$, $x \mapsto [x,y] = xyx^{-1}y^{-1}$ being an ind-morphism, it is in particular continuous. Since $H$ is obviously contained in $\varphi^{-1} ( \overline{D(H)} )$, we get $\overline{H} \subseteq \varphi^{-1} ( \overline{D(H)} )$. Consequently,   we have proven that
\[ \forall x \in \overline{H}, \; \forall y \in H, \; [x,y] \in \overline{D(H)}.\]
In turn (and analogously), for each fixed element $x$ of $\overline{H}$, the map $\psi \colon G \to G$, $y \mapsto [x,y]$ is continuous. Since $H$ is included into  $\psi^{-1} ( \overline{D(H)} )$, we get $\overline{H} \subseteq \psi^{-1} ( \overline{D(H)} )$ and thus
\[ \forall x,y \in \overline{H},  \; [x,y] \in \overline{D(H)}.\]
This  implies the desired inclusion.

(\ref{the closure of a solvable subgroup is a solvable subgroup}). If $H$ is solvable, it admits a sequence of subgroups such that
\[ H= H_0 \supseteq H_1 \supseteq \cdots \supseteq H_n = \{ 1 \} \quad \text{and} \quad D(H_i) \subseteq H_{i+1} \text{ for each } i.\]
This yields $\overline{H}= \overline{H}_0 \supseteq \overline{H}_1 \supseteq \cdots \supseteq \overline{H}_n = \{ 1 \}$ and by (\ref{the derivation of the closure is included into the closure of the derivation}) we get 
$ D( \overline{H}_i) \subseteq \overline{D( H_i)  } \subseteq \overline{H}_{i+1}$ for each $i$.
\end{proof}


\subsection{Automorphisms of the  affine $n$-space}

As usual, given an endomorphism $f\in\End(\A^n_{\kk})$, we denote by $f^*$ the corresponding endomorphism of the algebra of regular functions ${\mathcal O} (\A_{\kk}^n) = \kk [x_1 , \ldots ,  x_n]$. Note that every endomorphism $f\in\End(\A^n_{\kk})$ is uniquely determined by the polynomials $f_i=f^*(x_i)$,  $1\leq i\leq n$. 

In the sequel, we identify  the set ${\mathcal E}_n(\kk):= {\rm End} (\A^n_{\kk})$ with $(\kk [x_1, \ldots,x_n ])^n$. We thus simply denote by $f=(f_1,\ldots,f_n)$ the element of $\mathcal{E}_n(\kk)$ whose corresponding endomorphism $f^*$ is given by  
\[ f^* \colon {\mathcal O} (\A^n_{\kk}) \to  {\mathcal O} (\A^n_{\kk}), \quad P(x_1, \ldots,x_n) \mapsto P\circ f=P(f_1,\ldots,f_n). \]

The composition $g\circ f$ of two endomorphisms $f=(f_1,\ldots,f_n)$ and $g=(g_1,\ldots,g_n)$ is equal to
\[g\circ f=(g_1(f_1,\ldots,f_n),\ldots,g_n(f_1,\ldots,f_n)).\]

Note that for each nonnegative integer $d$, the following set is naturally an affine space (and therefore an algebraic variety!).
\[ \kk [x_1, \ldots,x_n]_{\leq \, d} := \{ P \in \kk [x_1, \ldots,x_n ], \; \deg P \leq d \}.\]
If $f= (f_1, \ldots,f_n) \in {\mathcal E}_n(\kk)$, we set $\deg f := \max_i \{\deg f_i\}$ and   define 
\[ {\mathcal E}_n(\kk)_{\leq \, d} := \{ f \in {\mathcal E}_n(\kk), \; \deg f \leq d \}. \]
The equality ${\mathcal E}_n(\kk)_{\leq \, d} = (  \kk [x_1, \ldots,x_n]_{\leq \, d})^n$ shows that ${\mathcal E}_n(\kk)_{\leq \, d}$ is naturally an affine space. Moreover, the filtration ${\mathcal E}_n(\kk) = \bigcup_d {\mathcal E}_n(\kk)_{\leq \, d}$ defines a structure of ind-variety on ${\mathcal E}_n(\kk)$.

We denote by ${\mathcal G}_n(\kk) =  \Aut (\A^n_{\kk})$  the automorphism group of $\A^n_{\kk}$. The next result allows us to endow ${\mathcal G}_n(\kk)$ with the structure of an ind-variety.

\begin{lemma}
Denote by ${\mathcal C}_n(\kk)$, resp.~${\mathcal J}_n(\kk)$, the set of elements $f$ in ${\mathcal E}_n(\kk)$ whose Jacobian determinant $\Jac (f)$ is a  constant, resp. a nonzero constant. Then, the following assertions hold:
\begin{enumerate}
\item \label{C is closed in E} The set ${\mathcal C}_n(\kk)$ is closed in ${\mathcal E}_n(\kk)$.
\item \label{J is closed in C} The set ${\mathcal J}_n(\kk)$ is open in ${\mathcal C}_n(\kk)$.
\item \label{G is closed in J} The set ${\mathcal G}_n(\kk)$ is closed in ${\mathcal J}_n(\kk)$.
\end{enumerate}
\end{lemma}

\begin{proof}
(\ref{C is closed in E}). Since $\deg ( \Jac (f) ) \leq n ( \deg (f) -1)$, the map $\Jac \colon {\mathcal E}_n(\kk) \to \kk [x_1, \ldots, x_n]$ is an ind-morphism. By definition, ${\mathcal C}_n(\kk)$ is the preimage of the set $\kk$ which is closed in $\kk [x_1, \ldots, x_n]$.

(\ref{J is closed in C}). The Jacobian morphism induces a morphism $\varphi \colon {\mathcal C}_n(\kk) \to \kk$, $f \mapsto \Jac (f)$.  By definition, ${\mathcal J}_n(\kk)$ is the preimage of the set $\kk^*$ which is open in $\kk $.

(\ref{G is closed in J}). Set ${\mathcal J}_{n,0}:= \{ f \in {\mathcal J}_n(\kk), \; f(0) = 0 \}$. Every element $f \in {\mathcal J}_{n,0}$ admits a formal inverse for the composition (see e.g.~\cite [Theorem 1.1.2]{Essen}), i.e.~a formal power series $g =\sum_{d \geq 1} g_d$, where each $g_d = (g_{d,1}, \ldots, g_{d,n})$ is a $d$-homogeneous element of ${\mathcal E}_n(\kk)$, meaning that $g_{d,1}, \ldots, g_{d,n}$ are $d$-homogeneous polynomials in $\kk[x_1, \ldots,x_n]$ such that
\[ f \circ g = g \circ f = (x_1, \ldots, x_n) \quad \text{(as formal power series).} \]
 Furthermore, for each $d$, the map $\psi_d \colon {\mathcal J}_{n,0} \to {\mathcal E}_n(\kk)$ sending $f$ onto $g_d$ is a morphism because each coefficient of every component of $g_d$ can be expressed as a polynomial in the coefficients of the components of $f$ and in the inverse $(\Jac f)^{-1}$ of the polynomial $\Jac f$. Recall furthermore (see \cite[Theorem 1.5]{BCW}) that every automorphism $f \in {\mathcal G}_n(\kk)$ satisfies
\begin{equation} \deg (f^{-1}) \leq  (\deg f)^{n-1}. \label{Gabber's inequality} \end{equation}
Therefore, an element $f \in {\mathcal J}_n(\kk)_{\leq \, d}$ is an automorphism if and if $\tilde{f} := f-f(0)$ is an automorphism. This    amounts to saying that $f$ is an automorphism if and only if $\psi_e ( \tilde{f}  ) = 0$ for all integers $e > d^{n-1}$. These conditions being closed, we have proven that ${\mathcal G}_n(\kk)_{\leq \, d}$ is closed in  ${\mathcal J}_n(\kk)_{\leq \, d}$ for each $d$, i.e. that ${\mathcal G}_n(\kk)$ is closed in ${\mathcal J}_n(\kk)$. Note that when the field $\kk$ has characteristic zero, the Jacobian conjecture (see for example \cite{BCW,Essen}) asserts that the equality ${\mathcal G}_n(\kk) = {\mathcal J}_n(\kk)$ actually holds.
\end{proof}

 Since the multiplication $\Gcal_n(\kk)\times\Gcal_n(\kk) \to\Gcal_n(\kk)$ and inversion $\Gcal_n(\kk)\to\Gcal_n(\kk)$ maps are morphisms (for  the inversion, this again relies on the fundamental inequality \eqref{Gabber's inequality}), we obtain that $\Gcal_n(\kk)$ is an ind-group.
\section{Borel subgroups}\label{section:Borel-subgps}


Throughout this section, we work over the field $\kk=\C$ of complex numbers.

Note that the affine subgroup 
\[\Acal_n=\{f=(f_1,\ldots,f_n)\in\Gcal_n(\C)\mid \deg(f_i)=1\text{ for all }i=1\ldots n\}\]
 and the Jonqui\`eres (or triangular) subgroup 
\begin{align*}
\Bcal_n & =  \{f=(f_1,\ldots,f_n) \in \Gcal_n(\C) \mid  \forall \, i, \; f_i =a_i x_i + p_i,  \; a_i \in \C^*, \; p_i \in \C[x_{i+1},\ldots,x_n]  \} \\
& =  \{f=(f_1,\ldots,f_n)\in\Gcal_n(\C)\mid  \forall \, i, \; f_i\in\C[x_{i},\ldots,x_n]  \} 
 \end{align*}
are both closed in $\Gcal_n(\C)$.

It is well known that the group $\Gcal_n(\C)$ is connected (see e.g.~\cite{Shafarevich2}*{proof of Lemma 4}, \cite{Kaliman}*{Proposition 2} or \cite{Popov2014}*{Theorem 6}). The same is true for $\Bcal_n$.

\begin{lemma}\label{lemma:G_n and B_n are connected} 
The groups $\Gcal_n(\C)=\Aut(\A^n_{\C})$ and $\Bcal_n$ are connected.
\end{lemma}

\begin{proof}
We say that a variety $V$ is  curve-connected if for all points $x,y \in V$, there exists a morphism $\varphi \colon C \to V$, where $C$ is a connected curve (not necessarily irreducible) such that $x$ and $y$ both belong to the image of $\varphi$. The same definition applies to ind-varieties. 

We   prove that  $\Gcal_n(\C)$ and $\Bcal_n$ are curve-connected. Let $f$ be an element in $\Gcal_n(\C)$. We first consider the morphism $\alpha \colon \A^1_{\C} \to \Gcal_n(\C)$ defined by 
\[\alpha (t) =f-tf(0,\ldots,0)\]
 which is contained in $\Bcal_n$ if $f$ is triangular. Note that $\alpha (0)=f$ and that the automorphism $\tilde{f}:=\alpha (1)$ fixes the origin of $\A^n_{\C}$. 

Therefore the morphism $\beta \colon \A^1_{\C} \smallsetminus \{ 0 \} \to \Gcal_n(\C)$, $t \mapsto (t^{-1}\cdot\id_{\A_{\C}^n})\circ\tilde{f}\circ (t\cdot\id_{\A_{\C}^n})$ extends to a morphism $\beta \colon \A^1_{\C} \to \Gcal_n(\C)$ (with values in $\Bcal_n$ if $f$, thus $\tilde{f}$, is triangular) such that $\beta (1) =\tilde{f}$ and such that $\beta (0) $ is a linear map, namely the linear part of $\tilde{f}$. This concludes the proof since $\textrm{GL}_n(\C)$ (resp.\ the set of all invertible upper triangular matrices) is curve-connected.
\end{proof}

Recall that the subgroup of upper triangular matrices in $\GL_{n} (\C)$ is solvable and has derived length 
$\left \lceil{\log_2(n)}\right \rceil +1$, where  $\left \lceil{x}\right \rceil $ denotes the smallest integer greater than or equal to   the real number $x$  (see e.g. \cite[page 16]{Wehrfritz}). In contrast, we have the following result.

\begin{lemma}\label{lemma:B_n is solvable}
The group $\Bcal_n$ is solvable of derived length $n+1$.
\end{lemma}

\begin{proof}
For each integer $k \in \{ 0, \ldots,n \}$, denote by $U_k$ the subgroup of $\Bcal_n$ whose elements are of the form $f= (f_1, \ldots, f_n)$ where $f_i = x_i$ for  all $i > k$ and $f_i = x_i + p_i$  with $p_i \in \C [x_{i+1}, \ldots,x_n ]$ for all  $i \leq k$. We will prove  $D (\Bcal_n ) = U_n$ and   $D^j (U_n) = U_{n-j}$ for all $j \in \{0, \ldots, n \}$. 

For this, we consider the dilatation $d(j,\lambda_j)$ and the elementary automorphism $e(j, q_j)$ which are defined for
every integer $j \in \{1, \ldots, n \}$, every nonzero constant $\lambda_j \in \C^*$ and every polynomial $q_j \in \C [x_{j+1}, \ldots,x_n ]$ by
\[ d(j,\lambda_j) = (g_1, \ldots,g_n)  \quad\text{and}\quad e(j, q_j) = (h_1, \ldots, h_n) ,\]
where  $g_j = \lambda_j x_j$, $h_j = x_j + q_j$ and $g_i = h_i = x_i$ for $i \neq j$. Note that an element $f\in U_k$ as above is equal to
\[ f = e(k,p_k) \circ \cdots \circ e(2,p_2) \circ e(1,p_1).\]
In particular, this tells us that $U_k$ is generated by the elements $e(j, q_j)$, $j \leq k$, $q_j \in \C [x_{j+1}, \ldots,x_n ]$.

The inclusion  $D (\Bcal_n ) \subseteq U_n$ is straightforward and left to the reader. The converse  inclusion  $ U_n   \subseteq D(\Bcal_n ) $ follows from the equality 
\[ [ e(j, q_j), d(j, \lambda_j) ] = e(j, (1- \lambda_j) q_j) .\]

Finally, we prove $D^j (U_n) = U_{n-j}$ by proving  that the equality $D (U_{k+1} ) = U_k$ holds for all $k \in \{ 0, \ldots, n-1 \}$.  The inclusion $D (U_{k+1} ) \subseteq U_k$ is straightforward and left to the reader. To prove  the converse inclusion, let us introduce the map $\Delta_i \colon \C [x_i, \ldots,x_n ] \to \C [x_i, \ldots,x_n]$, $q \mapsto q (x_i, \ldots, x_n) - q( x_i-1, x_{i+1}, \ldots, x_n)$. Note that $\Delta_i$ is surjective and that 
\[  [ e(j, q_j) , e(j+1, 1) ] = e(j, \Delta _{j+1} (q_j) )\]
for all $j \in \{1, \ldots, n-1 \}$ and all $q_j \in \C [x_{j+1}, \ldots, x_n]$.
This implies   $ U_k   \subseteq  D (U_{k+1} )$ and concludes the proof.
\end{proof}


\subsection{Triangular automorphisms form a Borel subgroup.}


In this section, we prove the first two statements of the theorem~\ref{main-theorem} from the introduction. For this, we need the following result.

\begin{proposition} \label{prop: the epsilon trick with the triangular group in any dimension}
Let $n \geq 2$ be an integer. If a closed subgroup of $\Aut (\A^n_{\C})$ strictly contains $\Bcal_n$, then it also contains at least one linear automorphism that is not triangular.
\end{proposition}

\begin{proof}
Let $H$ be a closed subgroup of $\Aut (\A^n_{\C})$ strictly containing $\Bcal_n$. We first prove that $H$ contains an automorphism whose linear part is not triangular. Let $f =(f_1, \ldots, f_n)$ be an element in $H\setminus \Bcal_n$. Then, there exists at least one component $f_i$ of $f$ that depends on an indeterminate $x_j$ with $j < i$, i.e.\ such that $\frac{\partial f_i}{\partial x_j} \neq 0$. Now, choose $c=(c_1,\ldots,c_n) \in \A^n_{\C}$ such that $\frac{\partial f_i}{\partial x_j}(c) \neq 0$ and consider the translation $t_c:=(x_1+c_1,\ldots,x_n+c_n)\in\Bcal_n$. Since
\[ f_i(x+c) = f_i (c) + \sum_k \frac{\partial f_i}{\partial x_k}(c) x_k + (\text{terms of higher order}),\]
 the linear part $l$ of $f \circ t_c$ is not triangular because  it corresponds to the (non-triangular) invertible matrix  $\left(  \frac{\partial f_i}{\partial x_k}(c) \right)_{ik}$. 
 Composing on the left hand side by another translation $t'$, we obtain an element $g:= t' \circ f \circ t \in H$ which fixes the origin of $\A^n_{\C}$ and whose linear part is again $l$.

For every  $\varepsilon\in\C^*$, set $h_{\varepsilon}:=(\varepsilon x_1,\ldots,\varepsilon x_n)\in\Bcal_n$. We can finally conclude by noting that
\[  \lim_{\varepsilon \to 0} h_{\varepsilon}^{-1}\circ g\circ h_{\varepsilon} = l \in H,\]
where the limit means that the ind-morphism $\varphi \colon \C^* \to  \Aut (\A^n_{\C})$, $\varepsilon \mapsto h_{\varepsilon}^{-1}\circ g\circ h_{\varepsilon}$ extends to a morphism $\psi \colon \C \to  \Aut (\A^n_{\C})$ such that $\psi (0) = l$. Since we have $\psi (\varepsilon) \in H$ for each $\varepsilon \in \C^*$, it is clear that $\psi (0)$ must also belong to $H$. Indeed, note that the set $\{ \varepsilon \in \C, \; \psi (\varepsilon) \in H \}$ is Zariski-closed in $\C$.
\end{proof}

\begin{proposition}\label{theorem:triangular automorphisms are a Borel subgroup}
Let $n \geq 2$ be an integer. Then, the Jonqui\`eres group $\Bcal_n$ is   maximal among all solvable subgroups of $\Aut ( \A^n_{\C})$.
\end{proposition}

\begin{proof}  Suppose by contradiction that there exists a  solvable subgroup $H$ of $\Aut (\A^n_{\C})$ that strictly contains $\Bcal_n$. Up to replacing $H$ by its closure $\overline{H}$ (see Lemma~\ref{lemma: the closure of a (solvable) subgroup is a (solvable) subgroup}), we may assume that $H$ is closed. By Proposition \ref{prop: the epsilon trick with the triangular group in any dimension}, the group $H \cap \Acal_n$ strictly contains $\Bcal_n \cap \Acal_n$. But since  $\Bcal_n \cap \Acal_n$ is a Borel subgroup of $\Acal_n$, this prove that $H \cap \Acal_n$ is not solvable, thus that $H$ itself is not solvable. Notice that we have used the fact that every Borel subgroup of a connected linear algebraic group is a maximal solvable subgroup. Indeed, every parabolic subgroup (i.e.\ a subgroup containing a Borel subgroup) of a connected linear algebraic group is necessarily closed and connected. See e.g.~\cite{Hum81}*{Corollary B of Theorem (23.1), page 143}.
\end{proof}

In dimension two, we establish another maximality property of the triangular subgroup which is actually stronger than the above one (see Remark \ref{rem:maximal-closed_implies_solvable-maximal} below).

\begin{proposition} \label{prop: B is a maximal closed subgroup of Aut}
The Jonquières group $\Bcal_2$ is  maximal among the closed subgroups of  $\Aut (\A^2_{\C})$.
\end{proposition}

\begin{proof} 
Let $H$ be a closed subgroup of $\Aut (\A^2_{\C})$ strictly containing $\Bcal_2$. By Proposition \ref{prop: the epsilon trick with the triangular group in any dimension} above, $H$ contains a linear  automorphism which is not triangular.   This implies that $H$ contains  all linear automorphisms, hence $\Acal_2$, and it is therefore equal to $\Aut ( \A^2_{\C})$. Recall indeed that  the subgroup $B_2=\Bcal_2\cap\GL_{2}(\C)$ of invertible upper triangular matrices  is a maximal subgroup of $\GL_{2}(\C)$, since the Bruhat decomposition  expresses $\GL_{2}(\C)$ as the disjoint union of  two double cosets of $B_2$, which are namely  $B_2$ and $B_2\circ f\circ B_2$, where $f$ is any element of $\GL_{2}(\C)\setminus B_2$.
\end{proof}

\begin{remark}
Proposition \ref{prop: B is a maximal closed subgroup of Aut} can not be generalized to higher dimension, since $\Bcal_n$ is strictly contained into the (closed) subgroup of automorphisms of the form  $f=(f_1,\ldots,f_n)$ such that $f_n=a_n x_n + b_n$ for some $a_n,b_n \in \C$ with $a_n \neq 0$.
\end{remark}

\begin{remark}\label{rem:maximal-closed_implies_solvable-maximal}
Proposition~\ref{prop: B is a maximal closed subgroup of Aut} implies Proposition~\ref{theorem:triangular automorphisms are a Borel subgroup} for $n=2$. Indeed, suppose that $\Bcal_2$ is strictly included into some solvable subgroup $H$ of $\Aut (\A^2_{\C})$. Up to replacing $H$ by $\overline{H}$ (see Lemma~\ref{lemma: the closure of a (solvable) subgroup is a (solvable) subgroup}), we may further  assume that $H$ is closed. By Proposition~\ref{prop: B is a maximal closed subgroup of Aut},   we would thus get that $H = \Aut (\A^2_{\C})$. But this is a contradiction because  the group $\Aut (\A^2_{\C})$ is obviously not solvable, since  it   contains the linear group $\textrm{GL}(2,\C)$ which is not solvable.
\end{remark}

By Proposition \ref{theorem:triangular automorphisms are a Borel subgroup}, we can say that the triangular group $\Bcal_n$ is a   Borel subgroup of $\Aut(\A_{\C}^n)$. This was already observed, in the case $n=2$ only, by Berest, Eshmatov and   Eshmatov in the nice paper  \cite{BEE} in which they obtained the following  strong results.  (In \cite{BEE}, these results  are stated for the group $\SAut(\A_{\C}^2)$ of polynomial automorphisms of $\A_{\C}^2$ of Jacobian determinant $1$, but all the  proofs remain valid  for $\Aut(\A_{\C}^2)$.)

\begin{theorem}[\cite{BEE}]\label{theorem:BEE}\begin{enumerate}
\item All Borel subgroups of $\Aut(\A_{\C}^2)$ are conjugate to $\Bcal_2$.
\item Every connected solvable subgroup of $\Aut(\A_{\C}^2)$ is conjugate to a subgroup of $\Bcal_2$.
\end{enumerate}
\end{theorem}

Recall that there exist, for every $n\geq3$, connected solvable subgroups of $\Aut(\A_{\C}^n)$ that are not conjugate to subgroups of $\Bcal_n$ \cite{Bass, Popov87}.  Hence, the second statement of the above theorem does not hold for $\Aut(\A_{\C}^n)$, $n\geq3$. 
 Similarly, we believe that not all Borel subgroups of $\Aut(\A_{\C}^n)$ are conjugate to $\Bcal_n$ if $n\geq3$. This would be clearly the case, if we knew that the following question has a positive answer. 

\begin{question}Is every connected solvable subgroup of $\Aut(\A_{\C}^n)$, $n\geq3$, contained into a maximal connected solvable subgroup?
\end{question}

The natural strategy to attack the above question would be  to apply  Zorn's lemma, as we do in the proof of the following general proposition. 

\begin{proposition} \label{proposition:Zorn}
Let $G$ be a group endowed with a topology. Suppose that there exists an integer $c>0$ such that every solvable subgroup of $G$ is of derived length at most $c$. Then, every solvable  (resp.\ connected solvable) subgroup of $G$ is contained into a maximal solvable  (resp.\ maximal connected solvable) subgroup.
\end{proposition}

\begin{proof}
Let $H$ be a solvable (resp.\ connected solvable) subgroup of $G$. Denote by $\Fcal$ the set of solvable (resp.\ connected solvable) subgroups of $G$ that contain $H$. Our hypothesis, on the existence of the bound $c$, implies that the poset $(\Fcal, \subseteq )$ is inductive. Indeed, if $(H_i)_{i \in I}$ is a chain in $\Fcal$, i.e.\ a totally ordered family of $\Fcal$, then the group $\bigcup_iH_i$ is solvable, because we have that
\[ D^j ( \bigcup_iH_i) = \bigcup_i D^j(H_i)\]
for each integer $j \geq 0$. Moreover, if all $H_i$ are connected, then so is their union. Thus, $\Fcal$ is inductive and we can conclude by Zorn's lemma.
\end{proof}

\begin{remark} Proposition \ref{proposition:Zorn} does not require any compatibility conditions between the group structure and the topology on $G$. Let us moreover recall that an algebraic group (and all the more an ind-group) is in general not a topological group.
\end{remark}

We are now left with another  concrete question.

\begin{definition}
Let $G$ be a group. We  set
\[ \psi (G) := \sup \{ l(H) \mid H \text{ is a solvable subgroup of } G \} \in \N \cup \{ + \infty \}, \]
where $l(H)$ denotes the derived length of $H$.
\end{definition}

\begin{question} \label{qu:borne}
Is   $\psi ( \Aut(\A_{\C}^n) )$ finite?
\end{question}

 Recall  that $\psi ( \textrm{GL}(n,\C) )$ is finite. This classical result has been  first established in 1937 by Zassenhaus \cite{Zassenhaus}*{Satz 7} (see also \cite{Malcev}). More recently, Martelo and Rib\'on have proved in \cite{Martelo-Ribon} that $\psi \big( ( \Ocal_{\rm ana} ( \C^n), 0) \big) < + \infty$, where $( \Ocal_{\rm ana} ( \C^n), 0)$ denotes the group of germs of analytic diffeomorphisms defined in a neighbourhood of the origin of $\C^n$.

Our next result answers Question \ref{qu:borne} in the case $n=2$.

\begin{proposition} \label{prop: bound for the derived length of solvable subgroups of Aut(A2)}
We have $\psi ( \Aut(\A_{\C}^2) ) = 5$.
\end{proposition}

\begin{proof}
The proof relies on  a precise description of all subgroups of $\Aut(\A_{\C}^2)$, due to Lamy, that we will recall below. Using this description, the equality $\psi ( \Aut(\A_{\C}^2) )=5$ directly follows from the equality  $\psi(\Acal_2)=5$ that we will establish in the next section (see Proposition~\ref{lem:psi(A2)}). The description of all subgroups of $\Aut(\A_{\C}^2)$ given by Lamy uses the amalgamated structure of this group, generally known as the theorem of Jung, van der Kulk and Nagata: The group $\Aut(\A_{\C}^2)$ is the amalgamated product of its subgroups $\Acal_2$ and $\Bcal_2$ over their intersection
\[ \Aut(\A_{\C}^2) = {\Acal_2} *_{\Acal_2 \cap \Bcal_2} \Bcal_2. \]
In the discussion below, we will use the Bass-Serre tree associated to this amalgamated structure. We refer the reader to \cite{Serre} for details on  Bass-Serre trees in full generality and to \cite{Lamy}  for details on the particular tree associated to the above amalgamated structure. That latter tree is the tree whose vertices are the left cosets $g \circ\Acal_2$ and $h \circ\Bcal_2$, $g,h \in \Aut(\A_{\C}^2)$. Two vertices $g \circ\Acal_2$ and $h \circ\Bcal_2$ are related by an edge if and only if there exists an element $k \in \Aut(\A_{\C}^2)$ such that $g \circ\Acal_2 = k \circ\Acal_2$ and $h \circ\Bcal_2 = k\circ \Bcal_2$, i.e. if and only if $g^{-1}\circ h \in \Acal_2 \circ \Bcal_2$. The group  $\Aut(\A_{\C}^2)$ acts on the Bass-Serre tree by left translation: For all $g,h \in  \Aut(\A_{\C}^2)$, we set $g. (h\circ \Acal_2) =(g\circ h)\circ \Acal_2 $  and  $g. (h \circ\Bcal_2) =(g\circ h)\circ \Bcal_2 $. Each element of $\Aut(\A_{\C}^2)$ satisfies one property of the following alternative:
\begin{enumerate}
\item It is triangularizable, i.e. conjugate to an element of $\Bcal_2$. This is the case where the automorphism fixes at least one point on the Bass-Serre tree.
\item It is a Hénon automorphism, i.e. it is conjugate to an element of the form
\[ g =a_1 \circ b_1 \circ \cdots \circ a_k \circ b_k,\]
where $k \geq 1$, each $a_i$ belongs to $\Acal_2 \smallsetminus \Bcal_2$ and each $b_i$ belongs to $\Bcal_2 \smallsetminus \Acal_2$. This is the case where the automorphism acts without fixed points, but preserves  a (unique) geodesic of the Bass-Serre tree  on which it acts as a translation of length $2k$.
\end{enumerate}
Furthermore, according to \cite{Lamy}*{Theorem 2.4}, every subgroup $H$ of $\Aut(\A_{\C}^2)$ satisfies one and only one of the following assertions:
\begin{enumerate}
\item It is conjugate to a subgroup of $\Acal_2$ or of $\Bcal_2$.
\item Every element of $H$ is triangularizable and $H$ is not conjugate to a subgroup of $\Acal_2$ or of $\Bcal_2$. In that case, $H$ is Abelian.
\item The group $H$ contains some Hénon automorphisms (i.e.\ non triangularizable automorphisms) and all those have the same geodesic on the Bass-Serre tree. The group $H$ is then solvable.
\item The group $H$ contains two Hénon automorphisms having different geodesics. Then, $H$ contains a free group with two generators.
\end{enumerate}
Let $H$ be now a solvable subgroup of $\Aut(\A_{\C}^2)$. If we are in case (1), then we may assume that $H$ is a subgroup of  $\Acal_2$ or of $\Bcal_2$. Since $ \psi ( \Acal_2) = 5$ and $\psi ( \Bcal_2) = 3$ (the group $\Bcal_2$ being  solvable of derived length $3$), this settles this case. In case $(2)$, $H$ is Abelian hence of derived length at most $1$. In case (3), there exists a geodesic $\Gamma$ which is globally fixed by every element of $H$. Therefore, we may assume without restriction that
\[ H= \{ f \in \Aut(\A_{\C}^2),  \; f ( \Gamma) = \Gamma \}.\]
Note that $D^2(H)$ is included into the group $K$ that fixes pointwise the geodesic $\Gamma$. Up to conjugation, we may assume that $\Gamma$ contains the vertex $\Bcal_2$, i.e.\ that $K$ is included into $\Bcal_2$. By \cite{Lamy}*{Proposition 3.3}, each element of $\Aut (\A^2_{\C})$ fixing an unbounded set of the Bass-Serre tree has finite order. If $f,g \in K$, their commutator is of the form $(x+ p(y), y +c)$. This latter automorphism being of finite order, it must be equal to the identity, showing that $K$ is Abelian. Therefore, we get $D^3(H) = \{ 1 \}$.

Finally, we cannot be in case (4), because a free group with two generators is not solvable.
\end{proof}

From Propositions \ref{proposition:Zorn} and \ref{prop: bound for the derived length of solvable subgroups of Aut(A2)}, we get at once the following result, which also follows from   Theorem  \ref{theorem:BEE} above.

\begin{corollary}\label{cor:connected into Borel}
Every solvable connected subgroup of $\Aut(\A_{\C}^2)$ is contained into a Borel subgroup.
\end{corollary}

\subsection{Proof of the equality  $\psi (\Acal_2) =5$.} 

Recall that Newman \cite{New} has   computed  the exact value $\psi ( \textrm{GL}(n,\C) )$ for all $n$. It turns out that $\psi ( \textrm{GL}(n,\C) )$ is equivalent to $5\log_9(n)$ as $n$ goes to infinity (see \cite{Wehrfritz}*{Theorem 3.10}). Let us give a few particular values for $\psi ( \textrm{GL}(n,\C) )$ taken from  \cite{New}.

\[
\begin{tabular}{|c|l|l|l|l|l|l|l|l|l|l|l|l|l|l|l|}
\hline
$n$  & $1$ & $2$ & $3$ & $4$ & $5$  & $6$ & $7$ & $8$ & $9$ & $10$ & $18$ & $26$ & $34$ & $66$ & $74$ \\
\hline
$\psi ( \textrm{GL}(n,\C) ) $  & $1$  & $4$ & $5$ & $6$ & $7$ & $7$ & $7$ & $8$ & $9$ & $10$ & $11$ & $12$ & $13$ & $14$ & $15$ \\
\hline
\end{tabular}
\]
\medskip

We now consider the affine group $\Acal_n$. On the one hand, observe that $\Acal_n$ is isomorphic to a subgroup of $\GL(n+1,\C)$. Hence, $\psi(\Acal_n)\leq\psi(\GL(n+1,\C))$. On the other hand, we have the short exact sequence
\[ 1 \to \C^n \to \Acal_n \xrightarrow{L}  \GL_n (\C )  \to 1,\]
where $L \colon \Acal_n  \to \GL (n, \C)$ is the natural morphism  sending an affine transformation to its linear part. 
Thus, if $H$ is a solvable subgroup of $\Acal_n$, we have a short exact sequence
\[ 1 \to H \cap (\C^n) \to H \xrightarrow{L}  L(H)  \to 1.\]
Since $L(H)$ is solvable of derived length at most $\psi ( \GL_n (\C) )$ and since $H \cap (\C^n)$ is Abelian, this implies that $l( H) \leq \psi ( \GL_n (\C ) ) + 1$.
Therefore, we have proved the general formula
\[\psi ( \GL_n (\C ) ) \leq\psi(\Acal_n)\leq\min\{\psi(\GL(n,\C))+1, \psi(\GL(n+1,\C))\}.\]
For $n=2$, this yields $\psi(\Acal_2)=4$ or $5$. We shall now prove that $\Acal_2$  contains  solvable subgroups of derived length 5 (see Lemma~\ref{lem: The derived length of G rtimes C2 is 5} below), hence the following desired result.

\begin{proposition} \label{lem:psi(A2)}
The maximal derived length of a solvable subgroup of the affine group $\Acal_2$ is $5$, i.e. we have $\psi(\Acal_2)=5$.
\end{proposition}

As explained above, it still remains to  provide an  example of   a solvable subgroup of $\Acal_2$ of derived length $5$.  In that purpose, recall that the group $\PSL (2, \C)$   contains a subgroup isomorphic to the symmetric group $S_4$ and that all such subgroups are conjugate (see for example \cite{Beauville}). 

\begin{definition}
The \emph{binary octahedral group} $ 2 {\rm O}$ is the pre-image of the symmetric group $S_4$ by the $(2:1)$-cover $\SL (2, \C) \to \PSL (2, \C)$.
\end{definition}

The following result is also well-known.

\begin{lemma} \label{lemma: length of the binary octahedral group}
The derived length of the binary octahedral group $G=  2 {\rm O}$ is $4$.
\end{lemma}

\begin{proof}
Using the short exact sequence
\[ 0 \to \{ \pm I \}  \to G \xrightarrow{\pi} S_4 \to 0, \]
we get $\pi ( D^2G) = D^2 (\pi (G) ) = D^2(S_4) = \V_4$, where $V_4 \simeq \Z_2 \times \Z_2$ is the Klein group. One could also easily check that $\pi^{-1} ( \V_4)$ is isomorphic to  the quaternion group ${\rm Q}_8$. The equality $\pi ( D^2G) = \V_4$ is then sufficient for showing that $D^2G = \pi^{-1} (\V_4)$. Indeed, if $D^2G$ was a strict subgroup of $\pi^{-1} (\V_4) \simeq {\rm Q}_8$, it would be cyclic, hence    $\pi (D^2G) = \V_4$ would be cyclic too. A contradiction. Since $D^2G \simeq {\rm Q}_8$ has derived length $2$, this shows us that the derived length of $G$ is $2+2 = 4$.
\end{proof}

\begin{lemma}  \label{lem: The derived length of G rtimes C2 is 5}
Consider the pre-image $L^{-1}(G)\simeq G \ltimes \C^2$ of the  binary octahedral group $G:= 2 {\rm O} \subseteq \SL (2, ÷C)$   by the natural morphism $ L \colon \Acal_2  \to \GL (2, \C)$ sending an affine transformation onto its linear part. Then, the derived length of   $L^{-1}(G)$   is equal to $5$.
\end{lemma}

\begin{proof}
By Lemma \ref{lemma: length of the binary octahedral group}, the derived length of $G$ is $4$. The short exact sequence $$1 \to \C^2 \to G \ltimes \C^2 \to G \to 1$$ implies that the derived length of $G \ltimes \C^2$ is at most $4+1 = 5$. Moreover, the strictly decreasing sequence $G=D^0(G) > D^1(G) > D^2(G) > D^3(G) > D^4(G) = 1$ shows that the group $D^2(G)$ is non-Abelian and in particular non-cyclic. By Lemma \ref{lem:derived-subgroups-of-Linverse} below, we thus have $D^i(G \ltimes \C^2) = D^i(G) \ltimes \C^2$ for every $i \leq 3$. But since $D^3(G)$ is non-trivial, the group $D^3(G \ltimes \C^2) = D^3(G) \ltimes \C^2$ strictly contains the subgroup $(\C^2, +)$ of translations and  cannot be Abelian, because the group $\C^2$ is its own centralizer in $\Acal_2$. Finally, we get $D^4(G \ltimes \C^2) \neq 1$, proving that the derived length of $G \ltimes \C^2$ is indeed $5$.
\end{proof}

\begin{lemma}\label{lem:derived-subgroups-of-Linverse}
Let $H$ be a finite non-cyclic subgroup of $\GL (2 , \C)$. Then the derived subgroup of $L^{-1}(H) = H \ltimes \C^2 \subseteq \Acal_2$ is the group $D(H) \ltimes \C^2$.
\end{lemma}

\begin{proof}
Set $K:= D(H \ltimes \C^2) \cap \C^2$. Note that $K$ contains the commutator $[ \id + v, h]$ for all $v \in \C^2$, $h \in H$, i.e.\ it contains all elements $h\cdot v-v$. It is enough to show that these vectors generate $\C^2$. Indeed, it would then imply that there exist $h_1,v_1,h_2,v_2$ such that the vectors $h_1\cdot v_1-v_1$ and $h_2\cdot v_2-v_2$ are linearly independent. But then, $K$ would also contain the vectors $h_1\cdot ( \lambda_1 v_1)-(\lambda_1v_1) + h_2\cdot ( \lambda _2 v_2) - \lambda_2v_2$ for any $\lambda_1, \lambda_2 \in \C$, proving that $K = \C^2$. Therefore, let us assume by contradiction that there exists a non-zero vector $w \in \C^2$ such that $h\cdot v-v$ is a multiple of $w$ for all $h \in H$, $v \in \C^2$. Take $w' \in \C^2$ such that $(w,w')$ is a basis of $\C^2$. In this basis, any element of $H$ admits a matrix of the form
\[ \left( \begin{array}{cc} a & b \\ 0 & 1 \end{array} \right). \]
Therefore, by the theory of representations of finite group, we may assume, up to conjugation, that each element of $H$ admits a matrix  of the form
\[ \left( \begin{array}{cc} a & 0 \\ 0 & 1 \end{array} \right).\]
This would imply  that $H$ is isomorphic to a finite subgroup of $\C^*$, hence that it is cyclic. A contradiction.
\end{proof}

\subsection{An ind-group with nonconjugate Borel subgroups.}

In this section, we consider the subgroup $\Aut_z(\A^3_{\C})$ of $\Aut(\A^3_{\C})$ of all automorphisms $f=(f_1,f_2,z)$  fixing the last coordinate of $\A^3_{\C}=\Spec(\C[x,y,z])$. Since it is clearly a closed subgroup, it is also an ind-group. 
Note that $\Aut_z(\A^3_{\C})$ is naturally isomorphic to a subgroup of $\Aut(\A^2_{\C(z)})$. In its turn, the field $\C (z)$ can be embedded into the field $\C$, so that the group $\Aut(\A^2_{\C(z)})$ is isomorphic to a subgroup of  $\Aut(\A^2_{\C })$. Therefore, by Proposition~\ref{prop: bound for the derived length of solvable subgroups of Aut(A2)}, we get
\[ \psi \big( \Aut_z(\A^3_{\C}) \big) \leq \psi \big( \Aut(\A^2_{\C(z)}) \big) \leq \psi \big( \Aut(\A^2_{\C }) \big) = 5.\] 
Recall moreover that $\Aut_z(\A^3_{\C})$ contains  nontriangularizable additive group actions \cite{Bass}. Let us briefly describe the  example given by Bass. Consider the following locally nilpotent derivation of $\C [x,y,z]$:
\[ \Delta = - 2 y \partial_x + z \partial_y. \]
Then, the derivation $(xz+y^2) \Delta$ is again locally nilpotent. We associate  it with the morphism
\[ (\C, +) \to \Aut _{\C} ( \C [x,y,z ] ), \quad t \mapsto \exp ( t (xz+y^2) \Delta). \]
The automorphism of $\A^3_{\C}$ corresponding to $ \exp ( t (xz+y^2) \Delta)$ is given by
\[ f_t:= (x-2t y (xz+y^2) -t^2 z(xz+y^2) ^2 , y + tz (xz+y^2) ,z) \in \Aut ( \A^3_{\C}).\]
For $t=1$, we get the famous Nagata automorphism. Note that the fixed point set of the corresponding $(\C,+)$-action on $\A^3_{\C}$ is the hypersurface $\{ xz+y^2 =0\}$ which has an isolated singularity at the origin. On the other hand, the fixed point set of a triangular $(\C,+)$-action on $\A^3_{\C}$
\[ t \mapsto g_t=  \exp ( t(a(y,z)\partial_x+b(z)\partial_y))    \in \Aut ( \A^3_{\C}) \]
is the set $\{ a(y,z) = b(z) = 0 \}$, which is isomorphic to a cylinder $\A^1_{\C}\times Z$ for some variety $Z$. This implies that the $(\C, +)$-action $t \mapsto f_t$  is not triangularizable.

By Proposition~\ref{proposition:Zorn}, it follows that $\Aut_z(\A^3_{\C})$ contains Borel subgroups that are not conjugate to a subgroup of the group \[\Bcal_z=\{(f_1,f_2,z)\in\Aut(\A^3_{\C})\mid f_1\in\C[x,y,z], f_2\in\C[y,z]\}\] 
of triangular automorphisms of $\Aut_z(\A^3_{\C})$.

\begin{proposition}\label{prop:Bz} The group $\Bcal_z$ is a Borel subgroup of $\Aut_z(\A^3_{\C})$. 
\end{proposition}

\begin{proof}With the same proof as for Lemma \ref{lemma:G_n and B_n are connected}, we obtain easily that $\Bcal_z$ is  connected. It is also solvable, since it can be seen as a subgroup of the Jonqui\`eres subgroup of $\Aut(\A^2_{\C(z)})$, which is solvable.

Now, we simply follow the proof of Proposition \ref{prop: the epsilon trick with the triangular group in any dimension}. Let $H\subset   \Aut_z(\A^3_{\C})$ be a closed subgroup containing strictly $\Bcal_z$ and take an element $f$ in $H\setminus\Bcal_z$ , i.e.\ an element $f=(f_1,f_2,z)$ with $f_2\in\C[x,y,z]\setminus\C[y,z]$. Arguing as before, we can find suitable translations $t_c=(x+c_1,y+c_2,z)$ and $t_{c'}=(x+c'_1,y+c'_2,z)$ such that the automorphism $g=t_c\circ f\circ t_{c'}$ fixes the point $(0,0,0)$ and  is of the form $g=(g_1,g_2,z)$ with $g_2=xc(z)+yd(z)+h(x,y,z)$ for some $c(z),d(z)\in\C[z]$, $c(z)\not\equiv0$, and some polynomial $h(x,y,z)$ belonging to the ideal $(x^2,xy,y^2)$ of $\C[x,y,z]$.     

Conjugating this  $g$ by the automorphism $(tx,ty,z)\in H$, $t\neq0$, and taking the limit when $t$ goes to 0, we obtain    an element of the form $(a(z)x+b(z)y, c(z)x+d(z)y,z)$ with 
$c(z)\not\equiv0$ in $H$. By Lemma \ref{lem:B2(C[z])} below, this implies that the group $H$ is not solvable.
\end{proof}

\begin{corollary} \label{corollary: Aut_z(A^3) admits non-conjugate Borel subgroups}
The ind-group $\Aut_z(\A^3_{\C})$ contains non-conjugate Borel subgroups.
\end{corollary}

In the course of the proof of Proposition \ref{prop:Bz}, we have used the following lemma that we prove now.

\begin{lemma}\label{lem:B2(C[z])}
The subgroup $\B_2 ( \C[z])$ of upper triangular matrices of $\GL_2(\C[z])$  is a maximal solvable  subgroup.
\end{lemma}

\begin{proof}
For every $\alpha \in \C$, denote by $\ev_{\alpha} \colon \GL_2(\C[z]) \to \GL_2(\C)$ the evaluation map that associates to an element $M(z) \in  \GL_2(\C[z])$ the constant matrix $M( \alpha )$ obtained by replacing $z$ by $\alpha$. Let $H$ be a subgroup of $\GL_2(\C[z])$ strictly containing the group $\B_2( \C [ z ] )$. By definition, $H$ contains a non-triangular matrix, i.e.\ a matrix of the form
\[ M =  \left( \begin{array}{cc} a(z) & b(z) \\ c(z) & d(z) \end{array} \right), \text{ with } c \not\equiv 0.  \]
Choose a complex number $\alpha$ such that $c ( \alpha ) \neq 0$. Then, the group $\ev_{\alpha} (H)$ contains the upper triangular constant matrices $\B_2 (\C)$ and a non-triangular matrix. Therefore,  $\ev_{\alpha} (H) = \GL_2( \C)$ and  $H$ is not solvable.
\end{proof}

\begin{remark}
By Nagao's theorem (see \cite{Nagao} or e.g. \cite[Chapter II, no 1.6]{Serre}), we have an amalgamated product structure
\[ \GL_2(\C[z]) = \GL_2 (\C) * _{\B_2( \C)} \B_2 (\C [ z ] ). \]
However, contrarily to the case of $\Aut ( \A^2)$, the group $\B_2 (\C [ z ] )$ is not a maximal closed subgroup. Indeed,  for every complex number $\alpha$, this group is strictly included into the group $\ev_{\alpha}^{-1}( \B_2 ( \C ) )$.
\end{remark}


\subsection{Maximal closed subgroups} \label{subsection: Maximal closed subgroups}


In this section, we mainly focus on the following question.

\begin{question}\label{qu:maximal closed subgroups}
What are the   maximal  closed subgroups of  $\Aut (\A^n_{\C})$? 
\end{question}

First of all,  it is easy to observe that, since the action of $\Aut(\A^n_{\C})$ on $\A^n_{\C}$ is infinite transitive, i.e.~$m$-transitive for all integers $m\geq1$, the stabilizers of a finite number of points are examples of  maximal closed subgroups. 

\begin{proposition}
For every finite subset $\Delta$ of $\A^n_{\C}$, $n \geq 2$, the group
\[ \Stab (\Delta ) = \{ f \in \Aut ( \A^n_{\C}), \; f( \Delta) = \Delta \}\]
is a maximal subgroup of $\Aut ( \A^n_{\C})$. Furthermore, it is closed.
\end{proposition}

\begin{proof}
Let $\Delta= \{ a_1,\ldots, a_k\}$ be a finite subset of $\A^n_{\C}$. Let $f \in \Aut(\A^n_{\C})\setminus\Stab(\Delta)$. We will prove that $\langle \Stab(\Delta), f \rangle = \Aut(\A^n_{\C})$, where $\langle \Stab(\Delta), f \rangle$ denotes the subgroup of $\Aut(\A^n_{\C})$ that is generated by $\Stab(\Delta)$ and $f$. We will use repetitively the well-known fact that $\Aut(\A^n_{\C})$ acts $2k$-transitively on $\A^n_{\C}$.

We first observe that $\langle \Stab(\Delta), f \rangle$ contains an element $g$ such that $g(\Delta)\cap\Delta=\emptyset$. To see this, denote by $m:=\left\vert{\Delta\cap f(\Delta)}\right\vert$  the cadinality of the set $\Delta\cap f(\Delta)$. Up to composing it by an element of  $\Stab(\Delta)$, we can suppose that $f$ fixes the points $a_1, \ldots, a_m$ and maps $a_{m+1}, \ldots, a_k$ outside $\Delta$. If $m\geq1$, then we consider an element $\alpha \in \Stab(\Delta)$ that maps the point $a_m$ onto $a_{m+1}$ and sends all points $f(a_{m+1}), \ldots, f(a_k)$ outside the set $f^{-1}(\Delta)$. Remark that  $g=f\circ\alpha\circ f$ is an element of $\langle \Stab(\Delta), f \rangle$ with $\left\vert{\Delta\cap g(\Delta)}\right\vert<m$. By descending induction on $m$, we can further suppose that $\left\vert{\Delta\cap g(\Delta)} \right \vert=0$ as desired.

Now, consider any $\varphi \in \Aut(\A^n_{\C})$. Let us prove that $\varphi$ belongs to the subgroup $\langle \Stab(\Delta), g \rangle$. Take  an element $\beta\in\Stab(\Delta)$ such that $\beta(\varphi(\Delta))\cap g^{-1}(\Delta)=\emptyset$. Then, $g(\beta(\varphi(\Delta))\cap \Delta=\emptyset$ and we can find an element $\gamma\in\Stab(\Delta)$ such that $(\gamma\circ g\circ\beta\circ\varphi)(a_i)=g(a_i)$ for all $i$. We have $\varphi= \beta^{-1}\circ g^{-1}\circ\gamma^{-1}\circ g\circ\delta \in \langle \Stab(\Delta), g \rangle$, where $\delta:=g^{-1}\circ(\gamma\circ g\circ\beta\circ\varphi)$ is an element of $\Stab(\Delta)$, proving that $\langle \Stab(\Delta), g \rangle$ is equal to the whole group $\Aut(\A^n_{\C})$. Therefore, the group $\Stab(\Delta)$ is actually maximal in $\Aut(\A^n_{\C})$. Finally, note that for each point $a \in \A^n_{\C}$ the evaluation map $\ev_a \colon \Aut(\A^n_{\C}) \to \A^n_{\C}$, $f \mapsto f(a)$ is and ind-morphism. Since $\Delta$ is a closed subset of $\A^n_{\C}$ the equality 
\[ \Stab (\Delta) = \bigcap_i (\ev_{a_i})^{-1} (\Delta) \]
implies that $\Stab (\Delta)$ is closed in $\Aut(\A^n_{\C})$.
\end{proof}

Besides the above examples and the triangular subgroup $\Bcal_2$, the only other maximal closed  subgroup  of  $\Aut (\A^2_{\C})$ that we are aware of is the affine subgroup $\Acal_2$. The fact that $\Acal_2$ is maximal among all closed subgroups of $\Aut (\A^2_{\C})$  is a particular case of the following recent result of Edo \cite{Edo}. (We recall that the so-called \emph{tame subgroup} of $\Aut (\A^2_{\C})$ is its subgroup generated by $\Acal_n$ and $\Bcal_n$.)

\begin{theorem}[\cite{Edo}] \label{theorem: Edo's result}
If a closed subgroup of $\Aut(\A^n_{\C})$, $n\geq2$, contains strictly the affine subgroup $\Acal_n$, then it also contains the whole tame subgroup, hence its closure. In particular, for $n=2$, the affine group $ \Acal_2$ is maximal among the closed subgroups of  $\Aut (\A^2_{\C})$.
\end{theorem}

\begin{remark}Note that Theorem \ref{theorem: Edo's result} does not allow us to settle the question of the (non) maximality of $\Acal_n$ among the closed subgroups of $\Aut(\A^n_{\C})$ when $n\geq3$. Indeed, on the one hand, it was recently shown that, in dimension 3, the tame subgroup is not closed (see \cite{Edo-Poloni}). But, on the other hand, it is still unknown whether it is dense in $\Aut(\A^3_{\C})$ or not. For $n\geq4$, the three questions, whether the tame subgroup is closed, whether it is dense, or even whether it is a strict subgroup of $\Aut(\A^n_{\C})$, are all open.  
\end{remark}

Let us finally remark that    the affine group $\Acal_2$ is not a maximal among all abstract subgroups of $\Aut(\A^2_{\C})$. Indeed, using the amalgamated structure
\[ \Aut(\A_{\C}^2) = {\Acal_2} *_{\Acal_2 \cap \Bcal_2} \Bcal_2 \]
and following \cite{Friedland-Milnor}, we can define the multidegree (or polydegree) of any automorphism $f \in \Aut(\A^2_{\C})$ in the following way. If $f$ admits an expression
\[ f = a_1 \circ b_1 \circ \cdots \circ a_k \circ b_k \circ a_{k+1},\]
where each $a_i$ belongs to $\Acal_2$, each $b_i$ belongs to $\Bcal_2$ and $a_i \notin \Bcal_2$ for $2 \leq i \leq k$, $b_i \notin \Acal_2$ for $1 \leq i \leq k$, the multidegree of $f$ is defined as the finite sequence (possibly empty) of integers at least equal to $2$:
\[ \mdeg (f) = ( \deg b_1, \deg b_2, \ldots, \deg b_k).\]
Then, the subgroup $M_r:= \langle \Acal_2, (\Bcal_2)_{\leq r} \rangle \subseteq \Aut(\A^2_{\C})$ coincides with the set of automorphisms whose multidegree is of the form $(d_1,\ldots, d_k)$ for some $k$ with $d_1,\ldots,d_k \leq r$. We thus have a strictly increasing sequence of subgroups
\[  \Acal_2= M_1 < M_2 < \cdots < M_d < \cdots,\]
showing in particular that $\Acal_2$ is not a maximal abstract subgroup.

\section{Non-maximality of the Jonqui\`eres subgroup in dimension 2} \label{section:dim2}

Throughout this section, we work over an arbitrary ground field $\kk$. 

Recall that by the famous Jung-van der Kulk-Nagata theorem \cites{Jung, Kulk, Nagata}, the group $\Aut(\A^2_{\kk})$, of algebraic automorphisms of the affine plane, is the amalgamated free product of its affine subgroup
\[A = \{ (ax+by+c, a'x+b'y +c')\in\Aut(\A^2_{\kk}) \mid a,b,c,a',b',c' \in \kk\}\]
and its Jonquières subgroup 
\[B:= \{ (ax+p(y), b'y+c')\in\Aut(\A^2_{\kk}) \mid a,b',c' \in \kk, p(y) \in \kk[y] \}\]
above their intersection. Therefore, every element $f\in\Aut(\A^2_{\kk})$ admits a \emph{reduced expression}  as a product of the form
\begin{equation}\tag{$\ast$}\label{equ:mot}f=t_1\circ a_1\circ t_2\circ\cdots\circ a_n\circ t_{n+1},\end{equation}
where $a_1,\ldots,a_n$ belong to $A\setminus A\cap B$, and $t_1,\ldots,t_{n+1}$ belong to $B$ with $t_2,\ldots,t_n\notin A\cap B$.

\begin{definition} \label{definition: affine length}
The number $n$ of  affine  non-triangular automorphisms appearing in such an expression  for $f$ is unique. We call it the \emph{affine length} of $f$ and denote it by $\ell_A (f)$.
\end{definition}

\begin{remark}Instead of counting affine elements to define the length of an automorphism of $\A^2$, one can of course  also consider the Jonquières elements and define the triangular length $\ell_B(f)$ of every $f\in\Aut(\A^2_{\kk})$. Actually, this is the triangular length, that one usually uses in the literature. Let us in particular recall that  this length map $\ell_B:\Aut(\A^2_{\C}) \to \N$ is lower semicontinuous \cite{Fur02}, when considering $\Aut(\A^2_{\C})$ as an ind-group. Since  
\[ \ell_A (f) = \max_{b_1,b_2 \in B} \ell_B (b_1\circ f\circ b_2) -1\]
for every  $f \in \Aut(\A^2_{\kk})$
and since  the supremum of arbitrarily many lower semicontinuous maps is lower semicontinuous, we infer that $\ell_A$ has also this property.
\end{remark}

\begin{proposition} \label{prop: lower semicontinuity of l_A}
The affine length map $\ell _A \colon  \Aut(\A^2_{\C}) \to \N$ is lower semicontinuous.
\end{proposition}

The next result shows that the Jonquières subgroup is not a maximal subgroup of $\Aut(\A^2_{\kk})$.

\begin{proposition}\label{prop:non-maximality}Let $p \in \kk [y]$ be a polynomial that fulfils the  following property: 
\begin{equation}\tag{WG}\label{eq-def:wgp} \forall \,  \alpha, \beta, \gamma \in \kk, \; \deg [ p(y) - \alpha p ( \beta y + \gamma )] \leq 1 \Longrightarrow \alpha = \beta =1\text{ and } \gamma = 0,
\end{equation}
and consider the following elements of $\Aut(\A^2_{\kk})$:
\[ \sigma = (y,x), \quad t= (-x +p(y), y), \quad f = (\sigma\circ t) ^2\circ \sigma\circ (t\circ \sigma)^2.\]
Then, the subgroup generated by $B$ and $f$ is a strict subgroup of $\Aut(\A^2_{\kk})$, i.e.\ $\langle B, f\rangle\neq\Aut(\A^2_{\kk})$.
\end{proposition}

\begin{remark}Polynomials  satisfying the above property \eqref{eq-def:wgp} are called \emph{weakly general} in \cite{Furter-Lamy}, where a stronger notion of a general polynomial is also given (see \cite{Furter-Lamy}*{Definition 15, page 585}). In particular, by \cite{Furter-Lamy}*{Example 65, page 608}, the polynomial $q = y^5 +y^4$ is weakly general if $\kk$ is a field of characteristic zero.

 Moreover,  the polynomial  $q = y^{2p} - y^{2p-1}$ is weakly general if $\textrm{char}(\kk)=p>0$. This follows directly from the fact that   the coefficients of $y^{2p}$, $y^{2p-1}$ and  $y^{2p-2}$ in the polynomial $q(y) - \alpha q ( \beta y + \gamma )$ are equal to
$1-\alpha\beta^{2p}$, $1-\alpha\beta^{2p-1}$ and  $-\alpha\beta^{2p-2}\gamma$, respectively.
\end{remark}

\begin{proof}[Proof of Proposition \ref{prop:non-maximality}]
Remark that $\sigma$ and $t$, hence $f$, are  involutions. Therefore, every element $g\in \langle B, f\rangle$ can be written as 
\[g = b_1\circ f\circ b_2\circ f\circ \cdots\circ b_k\circ f\circ b_{k+1},\]
where the elements $b_i$ belong to $B$ and where we can assume without restriction that  $b_2, \ldots, b_k$  are different from the identity (otherwise, the expression for $g$ could  be shortened using that   $f^2 = \id$).

In order to prove the proposition, it is enough to show that no element $g$ as above is of affine-length  equal to 1. Note that $\ell_A(g)=0$ if $k=0$ and that $\ell_A(g)=\ell_A(f)=5$ if $k=1$. It remains to consider the case where $k\geq2$.

 For this, let us  define four subgroups $B_0,\ldots,B_3$ of $B$ by
\[ \begin{array}{l}
B_0 = B,\\
B_1 = A \cap B = \{ (a x + b y + c, b' y +c')\mid a,b,c,b',c' \in \kk, a,b' \neq 0 \}, \\
B_2 = (A \cap B) \cap [ \sigma\circ (A \cap B)\circ \sigma ] =  \{ (a x +c, b' y + c')\mid a,c,b',c' \in \kk, a,b' \neq 0 \}, \\
B_3 = \{ (x, y +c')\mid c' \in \kk \}.
\end{array} \]
Note that $B=B_0 \supseteq B_1 \supseteq B_2 \supseteq B_3$. We will now give a reduced expression of $u_i := (t\circ \sigma)^2\circ b_i\circ (\sigma\circ t) ^2$ for each $i \in \{ 2, \ldots, k \}$. We do it by considering successively the four following cases: 
\[1. \ b_i \in B_0 \smallsetminus B_1; \quad 2.  \ b_i \in B_1 \smallsetminus B_2; \quad 3. \ b_i \in B_2 \smallsetminus B_3; \quad  4. \ b_i \in B_3 \smallsetminus \{ \id \}.\]

\underline{Case 1:}  $b_i \in B_0 \smallsetminus B_1$.

Since $b_i \in B \smallsetminus A$, the element $u_i$ admits the following reduced expression
\[ u_i = (t\circ \sigma) ^2\circ b_i\circ ( \sigma\circ t )^2.\]

\underline{Case 2:} $b_i \in B_1 \smallsetminus B_2$.

Since $\widehat{b_i}:= \sigma\circ b_i\circ \sigma \in A \smallsetminus B$, the element $u_i$ has the following reduced expression
\[ u_i= t\circ \sigma\circ t\circ  \widehat{b_i}\circ t\circ \sigma\circ t .\]

\underline{Case 3:} $b_i \in B_2 \smallsetminus B_3$.

Let us check that $\overline{b_i}:= t\circ \sigma\circ b_i\circ \sigma\circ t \in B \smallsetminus A$. We are in the case where $b_i = (a x +c, b' y + c')$ with $(a,c,b') \neq (1,0,1)$. A direct calculation gives that
\[ \overline{b_i}= ( b' x + p ( a y + c) - b' p(y) - c', a y + c).\]
By the assumption made on $p$, we have that $\deg [ p ( a y + c) - b' p(y) ] \geq 2$, hence that $\overline{b_i} \in B \smallsetminus A$. Therefore $u_i$ admits the following reduced expression
\[ u_i =t\circ \sigma\circ \overline{b_i}\circ \sigma\circ t .\]

\underline{Case 4.} $b_i \in  B_3 \smallsetminus \{ \id \}$.

Let us check that $\widetilde{b_i}:= (t\circ \sigma)^2\circ b_i\circ (\sigma\circ t) ^2 \in B \smallsetminus A$. We are in the case where $b_i = (x , y + c')$ with $ c' \in \C^*$. Using the computation in case 3 with $(a,c,b') = (1,0,1)$,  we then obtain that
\[ \widetilde{b_i}= t\circ \sigma \circ (x-c', y) \circ \sigma\circ t = t \circ  (x,y-c') \circ t = (x + p(y-c') - p(y), y-c') \in B \smallsetminus A. \]
Therefore,  the element $u_i$ has the following reduced expression
\[ u_i = \widetilde{b_i}.\]

Finally we obtain a reduced expression for an element $g\in \langle B, f\rangle$  from the above study of cases, since we can express 
\begin{align*}
g &= b_1\circ f\circ b_2\circ f\circ \cdots\circ b_k\circ f\circ b_{k+1}\\
&=b_1\circ(\sigma\circ t)^2\circ\sigma\circ u_2\circ\sigma\circ\cdots\circ\sigma\circ u_{k}\circ\sigma\circ(t\circ\sigma)^2\circ b_{k+1}.
\end{align*}
In particular, observe that $\ell_A(g)\geq6$ if $k\geq2$. This concludes the proof.
\end{proof}

Note that the element $f$ such that $\langle B, f\rangle\neq\Aut(\A^2_{\kk})$, that we  constructed in Proposition \ref{prop:non-maximality}, is of affine-length $\ell_A(f)=5$. Our next result shows that 5 is precisely the minimal length for elements $f\in\Aut(\A^2_{\kk})\setminus B$ with that property. 

\begin{proposition}\label{prop:ell<4}
Suppose that $f\in\Aut(\A^2_{\kk})$ is an automorphism of affine length $\ell$ with $1\leq\ell\leq4$. Then, the subgroup generated by $B$ and $f$ is equal to the whole group $\Aut(\A^2_{\kk})$, i.e.\ $\langle B, f\rangle=\Aut(\A^2_{\kk})$.
\end{proposition}

In order to prove the above proposition, it is useful to remark that we can  impose extra conditions on  the elements $t_1,\ldots,t_{n+1},a_1,\ldots,a_n$ appearing in a reduced expression \eqref{equ:mot} of an automorphism $f\in\Aut(\A^2_{\kk})$. We do it in Proposition \ref{prop:mots-reduits} below. First, we need to introduce some notations.

\begin{notation}
In the sequel, we  will denote, as in the proof of Proposition~\ref{prop:non-maximality}, by $\sigma$ the involution 
\[\sigma=(y,x)\in\Aut(\A^2_{\kk})\]
 and by $B_2$ the subgroup 
\[B_2 =\{ (ax+c, b'y+c')\in\Aut(\A^2_{\kk}) \mid a,c,b',c'\in\kk\}\subset A\cap B. \]  
Moreover, we denote by $I$ the  subset 
\[I=\{ (-x+p(y), y)\in\Aut(\A^2_{\kk})\mid p(y)\in\kk[y], \deg p(y)\geq2\}\subset B\setminus A\cap B .\]
Note that the elements of $I$ are all involutions.  
\end{notation}

\begin{lemma}\label{lem:B_2}
The followings hold:
\begin{enumerate}
\item $B_2\circ\sigma=\sigma\circ B_2$.
\item $B\setminus A\cap B=I\circ B_2=B_2\circ I=B_2\circ I\circ B_2$.
\item $A\setminus A\cap B\subset(A\cap B)\circ \sigma\circ (A\cap B$).
\end{enumerate}
\end{lemma}

\begin{remark}\label{rem:sigma}
In particular, Assertion (3) implies that the group  generated by $\sigma$ and all triangular automorphisms is equal to the whole $\Aut(\A^2_{\kk})$, i.e.\ $\langle B, \sigma\rangle=\Aut(\A^2_{\kk})$.
\end{remark}

\begin{proof}
The first assertion is an easy consequence of the following equalities:
\[(ax+c,b'y+c')\circ\sigma=(ay+c,b'x+c')=\sigma\circ(b'x+c',ay+c).\]

Let us now prove the second assertion. It is easy to check that $I\circ B_2=B_2\circ I=B_2\circ I\circ B_2\subset B\setminus A\cap B$. On the other hand, let $f=(ax+p(y),b'y+c')$ be an element of $B\setminus A\cap B$. Then $f$ belongs to $I\circ B_2$, since we can write  
\[f=(-x+p(\frac{y-c'}{b'}),y)\circ(-ax,b'y+c').\]

It remains to prove the last assertion. For this, it suffices to write, given an element $f=(ax+by+c,a'x+b'y+c')$ of $A\setminus A\cap B$ with $a'\neq0$, that
\[f=(ax+by+c,a'x+b'y+c')=(x+\frac{a}{a'}y+c,y+c')\circ\sigma\circ(a'x+b'y, \frac{ba'-ab'}{a'}y).\]
\end{proof}

\begin{proposition}\label{prop:mots-reduits}Let $f\in\Aut(\A^2_{\kk})$ be an automorphism of affine length $\ell=n+1$ with $n\geq0$. Then there exist triangular automorphisms $\tau_1, \tau_2\in B$ and triangular involutions $i_1,\ldots, i_n\in I$ such that
\begin{equation}\label{equ:mot-reduit}\tag{$\ast\ast$}
f=\tau_1\circ\sigma\circ i_1\circ\sigma \circ \cdots\circ\sigma\circ i_n\circ\sigma\circ\tau_2.
\end{equation}
In particular, the inverse of $f$ is given by 
\[f^{-1}=\tau_2^{-1}\circ\sigma \circ i_n \circ \sigma \circ \cdots\circ\sigma\circ i_1\circ\sigma\circ\tau_1^{-1}.\]
\end{proposition}

\begin{proof}Let $f$ be an automorphism of affine length $\ell=n+1$. By definition,
\[f=t_1\circ a_1\circ t_2\circ\cdots\circ a_n\circ t_{n+1},\]
for some $a_1,\ldots,a_n\in A\setminus A\cap B$,  $t_1,t_{n+1}\in B$ and $t_2,\ldots,t_n\in B\setminus A\cap B$. 
Using Assertion (3) of Lemma \ref{lem:B_2}, we may replace every $a_i$ by  $\sigma$. The proposition then follows from Assertions (1) and (2) of Lemma \ref{lem:B_2}.
\end{proof}

We can now proceed to the proof of Proposition \ref{prop:ell<4}.

\begin{proof}[Proof of Proposition \ref{prop:ell<4}]

\underline{Case $\ell=1$}. Let $f\in B$ with $\ell_A(f)=1$. By Proposition \ref{prop:mots-reduits}, we can write $f=\tau_1\circ\sigma\circ\tau_2$ for some $\tau_1, \tau_2\in B$. Thus, $\langle B, f\rangle=\langle B, \sigma\rangle=\Aut(\A^2_{\kk})$ follows from Remark \ref{rem:sigma}.

The proofs for affine length $\ell=2,3,4$ will be based on explicit computations. In particular, it will be useful to observe that all $i=(-x+p(y),y)\in I$ satisfy that
\begin{equation}\label{equ:1}
 i\circ(x+1,y)\circ i=(x-1,y),
\end{equation}

\begin{equation}\label{equ:2}
 \sigma\circ i\circ(x+1,y)\circ i\circ\sigma=(x,y-1)
\end{equation}
and
\begin{equation}\label{equ:3}
 i\circ(x,y-1)\circ i\circ(-x,y+1)=(-x+(p(y)-p(y+1)),y).
\end{equation}

\underline{Case $\ell=2$}. Let  $f\in B$ with $\ell_A(f)=2$. By Proposition \ref{prop:mots-reduits}, we can
suppose that $f=\sigma\circ i\circ\sigma$ for some involution $i=(-x+p(y),y)\in I$. Consider the elements $b_1=\sigma\circ(x,y-1)\circ\sigma$ and $b_2=\sigma\circ(-x,y+1)\circ\sigma$ of $B_2$. Since 
\[f\circ b_1\circ f\circ b_2= \sigma\circ i\circ(x,y-1)\circ i\circ(-x,y+1)\circ\sigma,\] 
it follows from Equality \eqref{equ:3} above that the automorphism $\sigma\circ(-x+(p(y)-p(y+1),y)\circ\sigma$ belongs to $\langle B, f\rangle$. By induction, we thus obtain an element in $\langle B, f\rangle$ of the form $\sigma\circ(-x+q(y),y)\circ\sigma$ with $\deg(q)=1$. This element is in fact an element of  $A\setminus A\cap B$ and has therefore affine length $1$. This implies that $\langle B, f\rangle=\Aut(\A^2_{\kk})$.

\underline{Case $\ell=3$}. Let  $f\in B$ with $\ell_A(f)=3$. By Proposition \ref{prop:mots-reduits}, we can
suppose that $f=\sigma\circ i_1\circ\sigma\circ i_2\circ\sigma$ for some  $i_1=(-x+p_1(y),y), i_2=(-x+p_2(y),y)\in I$. We first use Equality \eqref{equ:2}, which implies that  
\begin{equation}\label{equ:4}
\sigma\circ i_2\circ\sigma\circ b\circ\sigma\circ i_2\circ\sigma=(x,y-1),
\end{equation}   
where $b$ denotes the element $b=\sigma\circ(x+1,y)\circ\sigma\in B_2$ .
Hence, denoting by $b'$ the element $b'=\sigma\circ(-x,y+1)\circ\sigma$ in $B_2$ and using Equalities \eqref{equ:3} and \eqref{equ:4}, we obtain that 
\begin{align*}
f\circ b\circ f^{-1}\circ b' &=\sigma\circ i_1\circ\sigma\circ i_2\circ\sigma\circ b\circ \sigma\circ i_2\circ\sigma\circ i_1\circ\sigma\circ b'\\
&=\sigma\circ i_1\circ(x,y-1)\circ i_1\circ\sigma\circ b'\\
&=\sigma\circ i_1\circ(x,y-1)\circ i_1\circ(-x,y+1)\circ\sigma\\
&=\sigma\circ(-x+(p_1(y)-p_1(y+1)),y)\circ\sigma
\end{align*}
is an element of affine length $2$ (or $1$ in the case where $\deg(p_1)=2$), which belongs to  $\langle B, f\rangle$. Consequently, $\langle B, f\rangle=\Aut(\A^2_{\kk})$.

\underline{Case $\ell=4$}. Let  $f\in B$ with $\ell_A(f)=4$. By Proposition \ref{prop:mots-reduits}, we can
suppose that $f=\sigma\circ i_1\circ\sigma\circ i_2\circ\sigma\circ i_3\circ\sigma$ for some  $i_j=(-x+p_j(y),y)\in I$, $j=1,2,3$. Letting $b=\sigma\circ(x+1,y)\circ\sigma$ as above, one get that
\begin{align*}
f\circ b\circ f^{-1} &= \sigma\circ i_1\circ\sigma\circ i_2\circ\sigma\circ i_3\circ\sigma\circ b\circ \sigma\circ i_3\circ\sigma\circ i_2\circ\sigma\circ i_1\circ\sigma\\
&=\sigma\circ i_1\circ\sigma\circ i_2\circ(x,y-1)\circ i_2\circ\sigma\circ i_1\circ\sigma\\
&= \sigma\circ i_1\circ\sigma\circ i_2\circ(x,y-1)\circ i_2\circ(-x,y+1)\circ(-x,y-1)\circ\sigma\circ i_1\circ\sigma\\
&= \sigma\circ i_1\circ\sigma\circ i_2'\circ(-x,y-1)\circ\sigma\circ i_1\circ\sigma\\
&= \sigma\circ i_1\circ\sigma\circ i_2'\circ\sigma\circ(x-1,-y)\circ i_1\circ\sigma\\
&= \sigma\circ i_1\circ\sigma\circ i_2'\circ\sigma\circ i_1'\circ(x+1,-y)\circ\sigma\\
&= \sigma\circ i_1\circ\sigma\circ i_2'\circ\sigma\circ i_1'\circ\sigma\circ(-x,y+1),
\end{align*}
where $i_2'=(-x+p_2'(y),y)$ and $i_1'=(-x+p_1'(y),y)$ for the polynomials $p_2'(y)=p_2(y)-p_2(y+1)$ and $p_1'(y)=p_1(-y)$, respectively. In particular, $\langle B, f\rangle$ contains the element $\sigma\circ i_1\circ\sigma\circ i_2'\circ\sigma\circ i_1'\circ\sigma$. Since $\deg(p_2')=\deg(p_2)-1$, we obtain by induction an element in $\langle B, f\rangle$ of the form $\sigma\circ i_1\circ\sigma\circ \widetilde{i_2}\circ\sigma\circ i_1'\circ\sigma$ with $\widetilde{i_2}=(-x+\widetilde{p_2}(y),y)$ and $\deg(\widetilde{p_2})=1$. Since $\sigma\circ \widetilde{i_2}\circ\sigma$ is an element of $A\setminus A\cap B$, the above $\sigma\circ i_1\circ\sigma\circ \widetilde{i_2}\circ\sigma\circ i_1'\circ\sigma$ is an automorphism of affine length $3$, and the proposition follows.
\end{proof}

To conclude, let us emphasize that, as pointed to us by S.~Lamy,  our results concerning the non-maximality of $B$ are related to those  of \cite{Furter-Lamy} about the existence of normal subgroups for the group $\SAut(\A^2_{\C})$  of automorphisms of the complex affine plane whose Jacobian determinant is equal to $1$. Indeed, the subgroup $\langle B, f\rangle$, generated by $B$ and a given automorphism $f$, is contained into the subgroup $B\circ\langle f\rangle_N=\{h\circ g\mid h\in B, g\in\langle f\rangle_N\}$, where $\langle f\rangle_N$ denotes the normal subgroup of $\Aut(\A^2_{\C})$ that is generated by $f$.

Combined with Proposition~\ref{prop:ell<4},  the above observation gives us a short proof of the following result.

\begin{theorem}[\cite{Furter-Lamy}*{Theorem 1}] If $f \in \SAut(\A^2_{\C} )$ is of affine length at most 4 and $f\neq\id$, then the normal subgroup $\langle f\rangle_N$ generated by $f$ in $\SAut(\A^2_{\C} )$ is equal to the whole group  $\SAut( \A^2_{\C} )$.
\end{theorem} 

\begin{proof} The case where $f$ is a triangular automorphism being easy to treat (see  \cite{Furter-Lamy}*{Lemma 30, p. 590}), suppose that $f \in \SAut(\A^2_{\C} )$ is of affine length at most 4 and at least 1. By Proposition~\ref{prop:ell<4}, we have $\langle B, f\rangle= \Aut ( \A^2_{\C} )$. Since the group $B \circ\langle f\rangle_N$ contains $B$ and $f$, we get  $B\circ\langle f\rangle_N =\Aut(\A^2_{\C})$. In particular, the element $(-y,x)$ can be written as $(-y,x) = b \circ g$ for some $b \in B$ and $g \in \langle f \rangle_N$. Consequently, $\langle f\rangle_N$ contains the element $g=b ^{-1}\circ (-y,x) $ which is of affine length $1$.

Remark that the Jacobian determinant of $b$ is equal to $1$. Therefore, we can write $b^{-1}=(ax+P(y), a^{-1}y+c)$ for some $a\in\C^{*}$, $c\in\C$ and $P(y)\in\C[y]$. Thus, $g$ is given by 
\[g=(-ay+P(x), a^{-1}x+c).\]
Next, we consider the translation $\tau=(x+1,y)$ and compute the commutator $[\tau,g]=\tau\circ g\circ\tau^{-1}\circ g^{-1}$, which is an element of $\langle f\rangle_N$.  Since
\begin{align*}
[\tau,g] &=(x+1,y)\circ (-ay+P(x), a^{-1}x+c)\circ (x+1,y)\circ (ay-ac, -a^{-1}x+a^{-1}P(ay-ac)) \\
&=(x-P(ay-ac)+P(ay-ac-1)+1,y-a^{-1})
\end{align*}  
is a triangular automorphism different from the identity, the theorem follows directly from \cite{Furter-Lamy}*{Lemma 30, p. 590}.
\end{proof}

On the other hand, we can retrieve the fact that  the Jonquières subgroup is not a maximal subgroup of $\Aut(\A^2_{\C})$ as a corollary of \cite{Furter-Lamy}*{Theorem 2}. Indeed, the latter  produces  elements $f\in\SAut(\A^2_{\C})$ of affine length $\ell_A(f)=7$ such that $\langle f\rangle_N\neq\SAut(\A^2_{\C})$. In particular, by  \cite{Furter-Lamy}*{Theorem 1} above, the identity is the only  automorphism of affine length smaller than or equal to 4 contained in $\langle f\rangle_N$. Therefore, since  $\langle B, f\rangle \subset B\circ\langle f\rangle_N$, the subgroup $\langle B, f\rangle$ does not contain any non-triangular automorphism of affine length $\leq4$. Consequently, $\langle B, f\rangle$ is a strict subgroup of $\Aut(\A^2_{\C})$.



\begin{bibdiv}
\begin{biblist}

\bib{Bass}{article}{
   author={Bass, Hyman},
   title={A nontriangular action of ${\bf G}_{a}$ on ${\bf A}^{3}$},
   journal={J. Pure Appl. Algebra},
   volume={33},
   date={1984},
   number={1},
   pages={1--5},
}


\bib{BCW}{article}{
   author={Bass, Hyman},
   author={Connell, Edwin H.},
   author={Wright, David},
   title={The {J}acobian conjecture: reduction of degree and formal expansion of the inverse},
   journal = {Bull. Amer. Math. Soc. (N.S.)},
   volume = {7},
   year = {1982},
   number = {2},
   pages = {287--330},
}





\bib{Beauville}{article}{
   author={Beauville, Arnaud},
   title={Finite subgroups of ${\rm PGL}_2(K)$},
   conference={
      title={Vector bundles and complex geometry},
   },
   book={
      series={Contemp. Math.},
      volume={522},
      publisher={Amer. Math. Soc., Providence, RI},
   },
   date={2010},
   pages={23--29},
}





\bib{BEE}{article}{
   author={Berest, Yuri},
   author={Eshmatov, Alimjon},
   author={Eshmatov, Farkhod},
   title={Dixmier groups and Borel subgroups},
   journal={Adv. Math.},
   volume={286},
   date={2016},
   pages={387--429},
}



\bib{Edo}{article}{
   author={Edo, {\'E}ric},
   title={Closed subgroups of the polynomial automorphism group containing the affine subgroup},
    journal={Transform. Groups},
   date={2016},
   doi={10.1007/s00031-016-9412-7}
}




\bib{Edo-Poloni}{article}{
   author={Edo, {\'E}ric},
   author={Poloni, Pierre-Marie},
   title={On the closure of the tame automorphism group of affine
   three-space},
   journal={Int. Math. Res. Not. IMRN},
   date={2015},
   number={19},
   pages={9736--9750},
}

\bib{Essen}{book}{
   author={van den Essen, Arno},
   title={Polynomial automorphisms and the Jacobian conjecture},
   series={Progress in Mathematics},
   volume={190},
   publisher={Birkh\"auser Verlag, Basel},
   date={2000},
   pages={xviii+329},
}

\bib{Friedland-Milnor}{article}{
   author={Friedland, Shmuel},
   author={Milnor, John},
   title={Dynamical properties of plane polynomial automorphisms},
   journal={Ergodic Theory Dynam. Systems},
   volume={9},
   date={1989},
   number={1},
   pages={67--99},
}




\bib{Fur02}{article}{
   author={Furter, Jean-Philippe},
   title={On the length of polynomial automorphisms of the affine plane},
   journal={Math. Ann.},
   volume={322},
   date={2002},
   number={2},
   pages={401--411},
}







\bib{Furter-Lamy}{article}{
   author={Furter, Jean-Philippe},
   author={Lamy, St{\'e}phane},
   title={Normal subgroup generated by a plane polynomial automorphism},
   journal={Transform. Groups},
   volume={15},
   date={2010},
   number={3},
   pages={577--610},
}

\bib{Hum81}{book}{
   author={Humphreys, James E.},
   title={Linear algebraic groups},
   note={Graduate Texts in Mathematics, No. 21},
   publisher={Springer-Verlag, New York-Heidelberg},
   date={1975},
   pages={xiv+247},
}

\bib{Jung}{article}{
   author={Jung, Heinrich W. E.},
   title={\"Uber ganze birationale Transformationen der Ebene},
   language={German},
   journal={J. Reine Angew. Math.},
   volume={184},
   date={1942},
   pages={161--174},
}

\bib{Kaliman}{article}{
   author={Kaliman, Shulim},
   title={Isotopic embeddings of affine algebraic varieties into ${\bf
   C}^n$},
   conference={
      title={The Madison Symposium on Complex Analysis},
      address={Madison, WI},
      date={1991},
   },
   book={
      series={Contemp. Math.},
      volume={137},
      publisher={Amer. Math. Soc., Providence, RI},
   },
   date={1992},
   pages={291--295},
}


\bib{Kulk}{article}{
   author={van der Kulk, W.},
   title={On polynomial rings in two variables},
   journal={Nieuw Arch. Wiskunde (3)},
   volume={1},
   date={1953},
   pages={33--41},
}

\bib{Lamy}{article}{
   author={Lamy, St{\'e}phane},
   title={L'alternative de Tits pour ${\rm Aut}[{\mathbb C}^2]$},
   language={French, with French summary},
   journal={J. Algebra},
   volume={239},
   date={2001},
   number={2},
   pages={413--437},
}

\bib{Malcev}{article}{
   author={Mal{\cprime}cev, Anatoli\u{\i} Ivanovich},
   title={On certain classes of infinite soluble groups},
   journal={Amer. Math. Soc. Translations (2)},
   volume={2},
   date={1956},
   pages={1--21},
}


\bib{Martelo-Ribon}{article}{
   author={Martelo, Mitchael},
   author={Rib{\'o}n, Javier},
   title={Derived length of solvable groups of local diffeomorphisms},
   journal={Math. Ann.},
   volume={358},
   date={2014},
   number={3-4},
   pages={701--728},
}



\bib{Nagao}{article}{
   author={Nagao, Hirosi},
   title={On ${\rm GL}(2,\,K[x])$},
   journal={J. Inst. Polytech. Osaka City Univ. Ser. A},
   volume={10},
   date={1959},
   pages={117--121},
}




\bib{Nagata}{book}{
   author={Nagata, Masayoshi},
   title={On automorphism group of $k[x,\,y]$},
   note={Department of Mathematics, Kyoto University, Lectures in
   Mathematics, No. 5},
   publisher={Kinokuniya Book-Store Co., Ltd., Tokyo},
   date={1972},
   pages={v+53},
}



\bib{New}{article}{
    AUTHOR = {Newman, Mike F.},
     TITLE = {The soluble length of soluble linear groups},
   JOURNAL = {Math. Z.},
    VOLUME = {126},
      YEAR = {1972},
     PAGES = {59--70},
}


\bib{Popov87}{article}{
   author={Popov, Vladimir L.},
   title={On actions of ${\bf G}_a$ on ${\bf A}^n$},
   conference={
      title={Algebraic groups Utrecht 1986},
   },
   book={
      series={Lecture Notes in Math.},
      volume={1271},
      publisher={Springer, Berlin},
   },
   date={1987},
   pages={237--242},
}

\bib{Popov2014}{article}{
   author={Popov, Vladimir L.},
   title={On infinite dimensional algebraic transformation groups},
   journal={Transform. Groups},
   volume={19},
   date={2014},
   number={2},
   pages={549--568},
 }




\bib{Serre}{book}{
    AUTHOR = {Serre, Jean-Pierre},
     TITLE = {Trees},
    SERIES = {Springer Monographs in Mathematics},
      NOTE = {Translated from the French original by John Stillwell,
              Corrected 2nd printing of the 1980 English translation},
 PUBLISHER = {Springer-Verlag, Berlin},
      YEAR = {2003},
     PAGES = {x+142},
}



\bib{Shafarevich1}{article}{
   author={Shafarevich, Igor R.},
   title={On some infinite-dimensional groups},
   journal={Rend. Mat. e Appl. (5)},
   volume={25},
   date={1966},
   number={1-2},
   pages={208--212},
}

\bib{Shafarevich2}{article}{
   author={Shafarevich, Igor R.},
   title={On some infinite-dimensional groups. II},
   language={Russian},
   journal={Izv. Akad. Nauk SSSR Ser. Mat.},
   volume={45},
   date={1981},
   number={1},
   pages={214--226, 240},
}	


\bib{Wehrfritz}{book}{
   author={Wehrfritz, B. A. F.},
   title={Infinite linear groups. An account of the group-theoretic
   properties of infinite groups of matrices},
   note={Ergebnisse der Matematik und ihrer Grenzgebiete, Band 76},
   publisher={Springer-Verlag, New York-Heidelberg},
   date={1973},
   pages={xiv+229},
}




\bib{Zassenhaus}{article}{ 
    AUTHOR = {Zassenhaus, Hans},
     TITLE = {Beweis eines satzes \"uber diskrete gruppen},
   JOURNAL = {Abh. Math. Sem. Univ. Hamburg},
    VOLUME = {12},
      YEAR = {1937},
    NUMBER = {1},
     PAGES = {289--312},
}

\end{biblist}
\end{bibdiv}
\end{document}